\documentclass[letterpaper,11pt,reqno]{amsart}

\makeatletter
\usepackage{amssymb}
\usepackage{latexsym}
\usepackage{amsbsy}
\usepackage{amsfonts}
\usepackage{hyperref}
\usepackage{graphicx}
\usepackage{enumerate}
\usepackage{enumitem}
\usepackage{mathtools}
\usepackage{color}

\usepackage{tikz}

\def\marginpar#1{\ignorespaces}

\newtheorem{theorem}{Theorem}[section]

\numberwithin{equation}{section}
\makeatother
\begin{document}
\title[Occupation times]{The occupation time of a random walk generated by the uniform random permutation \\
-- a \texttt{GPT-6 Astra} proof}

\author[Wenpin Tang]{{Wenpin} Tang}
\address{Department of Industrial Engineering and Operations Research, Columbia University. 
} \email{wt2319@columbia.edu}

\date{\today} 
\begin{abstract}
This note provide a proof of Conjecture 2.5 in Fang et al. (J. Appl. Probab., 58(4):851--867, 2021),
with the help of \texttt{GPT-6 Astra}.
We show that the number of edges lying above zero of a random walk generated from the uniform permutation
has the discrete arcsine law.
The proof is short 
but hinges on an unexpected representation.
We also make several comments on solving the problem using large language models.
\end{abstract}

\maketitle


\bigskip
\quad The purpose of this note is to record a proof of Conjecture 2.5 in \cite{FG21},
thanks to \texttt{GPT-6 Astra}. 

\quad The problem is concerned with the distribution of the occupation time of a permutation generated random walk,
a model motivated by applications in genomics. \cite{WL18, WWH14}.
The problem and the subsequent paper \cite{FG21} grew out of the AIM workshop
 \url{https://aimath.org/pastworkshops/steinhd.html}.
 
\quad Let's describe the problem: let $\pi:=(\pi_1, \ldots, \pi_{n+1})$
be a permutation of $[n+1]: = \{1,\ldots, n+1\}$. 
Let 
\begin{equation*}
X_k = 
\left\{ \begin{array}{rcl}
+1 & \mbox{if } \pi_k < \pi_{k+1}, \\ 
-1 & \mbox{if }  \pi_k > \pi_{k+1},
\end{array}\right.
\end{equation*}
and denote by $S_n: = \sum_{k = 1}^n X_k$, $S_0:=0$
the corresponding walk generated by $\pi$.
An obvious candidate for $\pi$ is the uniform random permutation of $[n+1]$.
This random walk model also appeared in the physics literature \cite{OV04},
and in the study of zigzag diagrams \cite{GO06}.

\quad One intriguing quantity related to the walk $(S_k, \, 0 \le k \le n)$ is the occupation time above zero:
\begin{equation}
\label{eq:Nn}
N_n:= \sum_{k = 1}^n 1(S_{k-1} \ge 0, \, S_k \ge 0);
\end{equation}
that is, the number of edges which lie above zero up to time $n$.
The following result (now a theorem) was conjectured in \cite{FG21}.
\begin{theorem}
\label{thm:1}
For $\pi$ a uniform permutation of $[2n+1]$,
let $(S_k, \, 0 \le k \le 2n)$ be generated from $\pi$.
The distribution of $N_{2n}$ defined by \eqref{eq:Nn} is given by
\begin{equation}
\label{eq:arcsine}
\mathbb{P}(N_{2n} = 2k) = \frac{1}{2^{2n}} \binom{2k}{k} \binom{2n-2k}{n-k}
\quad \mbox{for } k =0,\ldots, n;
\end{equation}
that is, the discrete arcsine law. 
\end{theorem}

\quad Before proving Theorem \ref{thm:1}, let's make a few comments.
\begin{enumerate}[itemsep = 3 pt]
\item
The reason why we believed the formula \eqref{eq:arcsine} {\em should be} true is because
(a) it holds for all small cases from computer enumerations;
(b) the boundary case $k = 0, n$ is known via a non-trivial bijection \cite{BDN10}.
\item
As will be seen, the proof of Theorem \ref{thm:1} culminates an ``aha" moment. 
(It is short, and is eventually an exercise of good tricks.)
Our failure to find a proof 
was partly due to our firm insistence on a bijective proof in the spirit of \cite{BDN10},
``misled" by the boundary case.
\texttt{GPT-6 Astra} (Medium) was also not able to produce such a (bijective) proof.
\item
Theorem \ref{thm:1} is one of the three conjectures that I constantly (since \texttt{GPT-4}) used to test the reasoning
ability of large language models.
The other two are on one-dependent processes \cite{HL15},
 and Poisson binomial distributions \cite{GLP17, TT23}.
 All these conjectures are probabilistic/combinatorics,
 which I believe are true.
 To the date, only Theorem \ref{thm:1} has been resolved with \texttt{GPT-6 Astra}.
 (All the previous GPT versions also failed to prove Theorem \ref{thm:1}.)
\end{enumerate}

\quad Below is the \texttt{GPT-6 Astra}'s (1m51s) proof,
which is hard to beat.
\begin{proof}[Proof of Theorem \ref{thm:1}]
Since $(S_k, \, 0 \le k \le n)$ is a simple walk (with $\pm 1$ increments),
the values of $S_{2k-1}$ are odd ($\ne 0$).  
So the edges adjacent to $S_{2k-1}$, 
i.e., $(S_{2k-2}, S_{2k-1})$ and $(S_{2k-1}, S_{2k})$,
lie above zero if and only if $S_{2k-1} > 0$.
As a result,
\begin{equation}
\label{eq:01}
N_{2n} = 2 \sum_{k =1}^n 1(S_{2k-1} > 0),
\end{equation}
which holds pathwise.

\quad Next, a uniform random permutation can be constructed via i.i.d. uniform$(0, 1)$ random variables.
Let $U_1, \ldots, U_{2n+1} \sim \mbox{Unifom}(0,1)$, 
and sort these numbers in increasing order $U_{(1)} < \cdots < U_{(n)}$.
Then the permutation $\pi$ defined by
\begin{equation*}
\pi_k:= 1 + \sum_{j = 1}^{2n+1} 1(U_j < U_k), \quad k = 1, \ldots, 2n+1,
\end{equation*} 
(i.e., $U_{\pi_k} = U_{(k)}$)
 is uniform
(see e.g., \cite[Section 3.4.2]{Knuth}).
We have:
\begin{equation}
\label{eq:05}
X_k = 1 - 2 1(U_k > U_{k+1}) \quad \mbox{and hence}, \quad
S_k = k - 2 \sum_{j = 1}^k 1(U_{j} > U_{j+1}).
\end{equation}

\quad Now comes the ``aha" moment. 
Define 
\begin{equation*}
V_1 = U_1, \quad V_k = \{U_k - U_{k-1}\} \, \mbox{for } k \ge 2,
\end{equation*}
where $\{x\}$ denotes the fraction part of $x$.
The random variables $V_k$ are i.i.d. $\mbox{Unifom}(0,1)$.
Observe that
\begin{equation*}
V_{j+1} = U_{j+1} - U_j + 1(U_{j} > U_{j+1}),
\end{equation*}
so 
\begin{equation*}
\sum_{j = 1}^{k+1} V_j = U_{k+1} + \sum_{j = 1}^k 1(U_{j} > U_{j+1}).
\end{equation*}
Therefore, 
$\sum_{j = 1}^k 1(U_{j} > U_{j+1}) = \left[ \sum_{j = 1}^{k+1} V_j\right]$,
where $[x]$ denotes the integer part of $x$.
By \eqref{eq:05},
\begin{equation*}
S_k = k - 2 \left[ \sum_{j = 1}^{k+1} V_j\right].
\end{equation*}
So
\begin{equation}
\label{eq:06}
S_{2k-1} > 0 
\Longleftrightarrow
\sum_{j = 1}^{2k} V_j < k.
\end{equation}

\quad Finally, define the ``new" increments and the walk:
\begin{equation}
Y_k:= 1 - V_{2k-1} - V_{2k} \quad \mbox{and} \quad 
T_k:=\sum_{j = 1}^k Y_k, \quad \mbox{for } k = 1,\ldots,n.
\end{equation}
By \eqref{eq:01} and \eqref{eq:06},
we get:
\begin{equation}
N_{2n} = 2 \sum_{k=1}^n 1(T_k > 0)
= 2 \#\{k: T_k > 0\}.
\end{equation}
Note that the increments $Y_k$'s
are  i.i.d., continuous, and symmetric with respect to $0$.
By the classical Sparre–Anderson theorem 
(see e.g., \cite[Chapter XII.7]{Fellervol2}),
$\#\{k: T_k > 0\}$ has the discrete arcsine law,
which allows us to conclude.
\end{proof}

\bigskip
{\bf Acknowledgement:}
This research is supported by NSF CAREER Award DMS-2538791
and
NSF MFAI Award DMS-2602038.

\bibliographystyle{abbrv}
\bibliography{unique}
\end{document}